\documentclass[12pt]{article}
\usepackage{color}

\usepackage{amssymb}   
\usepackage{amsthm}    
\usepackage{amsmath}   
\usepackage{stmaryrd}  
\usepackage{titletoc}  
\usepackage{mathrsfs}  
\usepackage{graphicx}

\newlength{\defbaselineskip}
\newcommand{\setlinespacing}[1]%
           {\setlength{\baselineskip}{#1 \defbaselineskip}}

\theoremstyle{plain}
\newtheorem{thm}{Theorem}[section]
\newtheorem{cor}[thm]{Corollary}
\newtheorem{lem}[thm]{Lemma}
\newtheorem{prop}[thm]{Proposition}
\newtheorem{exam}[thm]{Example}

\theoremstyle{definition}
\newtheorem{defn}{Definition}[section]

\newtheorem{rmk}{Remark}[section]

\newcommand{\eps}{\varepsilon}

\newcommand{\cH}{\mathcal{H}}
\newcommand{\cL}{\mathcal{L}}

\newcommand{\cA}{\mathcal{A}}
\newcommand{\cS}{\mathcal{S}}

\newcommand{\bP}{\mathbb{P}}
\newcommand{\bR}{\mathbb{R}}
\newcommand{\bN}{\mathbb{N}}

\newcommand{\bV}{\mathbb{V}}
\newcommand{\bL}{\mathbb{L}}
\newcommand{\sF}{\mathscr{F}}
\newcommand{\sP}{\mathscr{P}}

\makeatletter\@addtoreset{equation}{section} \makeatother
 \allowdisplaybreaks

\begin{document}

\title{H\"ormander-Type Theorem for It\^o Processes and Related Backward SPDEs
}

\author{Jinniao Qiu\footnotemark[1]  }
\footnotetext[1]{Department of Mathematics, Humboldt-Universit\"at zu Berlin, Unter den Linden 6, 10099 Berlin, Germany. \textit{E-mail}: \texttt{qiujinn@gmail.com}. Financial support from the chair Applied Financial Mathematics  is gratefully acknowledged.}

%
%

\maketitle

\begin{abstract}
A H\"ormander-type theorem is established for It\^o processes and related backward stochastic partial differential equations (BSPDEs).   A short self-contained proof is also provided for the $L^2$-theory of linear, possibly degenerate BSPDEs, in which new gradient estimates are obtained.
\end{abstract}




\section{Introduction}

Let $(\Omega,\bar{\sF},\{\bar{\sF}_t\}_{t \geq 0},\bP)$ be a complete filtered probability space, on which two independent $d_1$-dimensional Weiner processes $W=(W_t)_{t\geq 0}$ and $B=(B_t)_{t\geq 0}$ are well defined.  The filtration generated by $W$, together with all $\bP$ null sets, is denoted by $\{\sF_t\}_{t\geq 0}$. The $\sigma$-algebra of the predictable sets on $\Omega\times[0,+\infty)$ associated with $\{\sF_t\}_{t\geq 0}$ is denoted by $\sP$, and $\sF:=\cup_{t\geq 0}\sF_t$.

An It\^o process (see \cite{Stroock-Varadhan-Multi-dim-diffun79}) starting from time $s $ and position $x$ is of the form
\begin{align}
X_t^{s,x}=x+\int_s^tb(r,X_r^{s,x})\,dr+\int_s^t\sigma^k(r,X_r^{s,x})\,dB^k_r
+
\int_s^t\theta^k(r,X_r^{s,x})\,dW^k_r, \,\,0\leq s\leq t.
\end{align}
Here and throughout this paper, the summation over repeated indices is enforced by convention unless stated otherwise. Fix $T\in(0,\infty)$ and define
\begin{align}
u(t,x)=E_{\bar{\sF}_t}\left[
\int_t^Tf(r,X_r^{t,x})\,dr
+G(X_T^{t,x})
\right],\quad (t,x)\in[0,T]\times\bR^d. \label{eq-conditn-exp}
\end{align}
For the sake of convenience, we assume that 

\medskip
$(\cA0)$ $b$, $\sigma$, $\theta$ and $f$ are $\sP\times \mathcal{B}(\bR^d)$-measurable and $G$ is $\sF_T\times\mathcal{B}(\bR^d)$-measurable. 
\rm\medskip

Under certain conditions (see Proposition \ref{prop-repsn} and Remark \ref{rmk-measurablity}), the random field $u$ turns out to be $\sP\times \mathcal{B}(\bR^d)$-measurable and together with some \textit{endogenous} random field $v$, it satisfies the following BSPDE
\begin{equation}\label{BSPDE}
  \left\{\begin{array}{l}
  \begin{split}
  -du(t,x)=\,&\displaystyle \left[ \frac{1}{2}(L_k^2+M_k^2)u(t,x)+M_kv^k(t,x)
         +\tilde{b}^j(t,x)D_ju(t,x)+f(t,x)
                \right]\, dt\\ &\displaystyle
           -v^{r}(t,x)\, dW_{t}^{r}, \quad
                     (t,x)\in [0,T]\times \bR^d;\\
    u(T,x)=\, &G(x), \quad x\in\bR^d,
    \end{split}
  \end{array}\right.
\end{equation}
where it is written in the H\"ormander form, $D=(D_1,\dots,D_d)$ is the gradient operator, $L_k=\sigma^{jk}D_j$, $M_k=\theta^{jk}D_j$, for $k=1,\dots,d_1$, and $\tilde{b}^j=b^j-\frac{1}{2}\left(\sigma^{ik}D_i\sigma^{jk}+\theta^{ik}D_i\theta^{jk}\right)$, for $j=1,\dots,d$. BSPDE like \eqref{BSPDE} is said to be degenerate when it fails to satisfy the following super-parabolicity: There exists $\lambda\in(0,\infty)$  such that
   \begin{align*}
       \sigma^{ik}\sigma^{jk}(t,x)\xi^i\xi^j\geq \lambda |\xi|^2\quad \text{a.s.,} \,\,\forall\, (t,x,\xi)\in [0,T]\times\bR^d\times\bR^d.
   \end{align*}

Borrowing notions from the optimal stochastic control theory, we say the \textit{framework} is Markovian if and only if all the coefficients $b$, $\sigma$, $\theta$, $f$ and $G$ are \textit{deterministic} functions. In the Markovian case, $X$ is a diffusion (\textit{Markovian}) process,  $u$ is deterministic, $v\equiv 0$, and BSPDE \eqref{BSPDE} turns out to be a classical parabolic PDE. In H\"ormander's seminal work \cite{hormander-HypoEllip1967}, it is proved with the analytical method that given smooth coefficients $b$, $\sigma$, $\theta$ and $f$, under the hypo-ellipticity condition allowing for degenerateness (like condition $(\cH)$ below), $u$ is smooth on $[0,T)\times \bR^d$, even when the terminal value $G$ is just a generalized (irregular) function. H\"ormander's theorem shows in fact the smoothness of transition probabilities of hypo-elliptic diffusions, for which the probabilistic approach was formulated on the basis of Malliavin calculus (see \cite{Malliavin-1978}), and along this line, see \cite{CassFriz-2010,Mattingly-Pardoux-2006} and references therein for the generalizations.  

In this paper, we are concerned with a H\"ormander-type theorem for It\^o processes which allow for random, possibly degenerate coefficients and go beyond the scope of Markovian framework and thus of diffusion processes. In fact, for It\^o processes, the smoothing property depends not only on the (hypo-)ellipticity of the diffusion coefficients but also on the extent to which the \textit{framework} is \textit{Markovian}. In other words, not only the degenerateness but also the randomness of coefficients may damage the smoothing property.  Let us consider the following example.
\begin{exam}\label{exam}
Let $d=d_1=1$, $\sigma\equiv 0$, $f\equiv0$, $\theta\equiv 1$, $b(t,\omega)=\bar{b}(t,H_t(\omega))$  and $G(x)=U(x-H_T)M_T$ with $\bar b$ and $U$ being deterministic functions, $H_t= X^{0,\eta}_t$ and $M_t=\exp\{\alpha W_t-\frac{\alpha^2t}{2}\}$ for $t\in[0,T]$, $\eta,\alpha\in \bR$. (In the field of mathematical finance, $X$ can be seen as a wealth process, the terminal value $G(x)=U(x-H_T)M_T$ is the utility from terminal wealth which is subject to the delivery of liability $H_T$, and $M_T$ denotes the transformation of probability measures.) 

It is easy to check that $u(t,x):=U(x-H_t)M_t$ along with $v(t,x):=\alpha U(x-H_t)M_t-U'(x-H_t)M_t$ solves BSPDE \eqref{BSPDE}. Moreover, we see that $u$ does not have more spacial regularity than its terminal value $G$. Taking a close look at the non-Markovian framework, we consider the two particular cases: \medskip \\ 
 (i) when $\eta=0$ and $b\equiv 0$, $X_t^{s,x}=x+W_t-W_s$ is Markovian and the framework is not Markovian due to the randomness of $G(x)=U(x-W_T)M_T$; \medskip \\
 (ii) when $\alpha=0$ and $H$ is chosen to be the Brownian bridge with $H_T=0$, then $X$ is equipped with a random drift and thus is not a Markov process while the terminal value $G(x)=U(x)$ is deterministic. 
\end{exam}

In view of assumption $(\cA 0)$, we see that the randomness of all the coefficients $b$, $\sigma$, $\theta$, $f$ and $G$ is only subject to the sub-filtration $\{\sF_t\}_{t\geq 0}$ that is generated by Wiener process $W$ and one may conjecture that the term associated with Wiener process $B$, seen as the Markovian part, may serve to the smoothing property. The answer is affirmative. Under a hypo-ellipticity assumption on the coefficients $\{\sigma^k\}$, $k=1,\dots,d_1$ (see $(\cH)$ below),  we prove that the random field $u(t,x)$ is  almost surely infinitely differentiable with respect to $x$ and each of its derivatives is continuous in $(t,x)$ on $[0,T)\times \bR^d$. 
Compared with the time-smoothness assumption in the classical H\"ormander theorem, the coefficients herein is only required to be measurable with respect to the time variable, and the time-differentiability of $u(t,x)$ is not investigated due to the appearance of the stochastic integral in BSPDE \eqref{BSPDE}.  For the related linear, possibly degenerate BSPDEs, a short self-contained proof is presented for the $L^2$-theory, and in particular, we obtain some new gradient estimates from which we start the proof of the H\"ormander-type theorem. 

Inspired by the filtering theory of partially observable diffusion processes,  Krylov \cite{krylov2013-Hormder-SPDE} has just obtained a H\"ormander-type theorem for \textit{forward} SPDEs. However, there is an essential difference between \textit{forward} SPDEs and BSPDEs, i.e., the noise term in the former is exogenous, while in the latter it comes from the martingale representation and is governed by the coefficients, and thus it is endogenous. On the other hand, we would also not that the method of Krylov \cite{krylov2013-Hormder-SPDE} relies on the generalized It\^o-Wentzell formula and associated results on deterministic PDEs, while we use directly elaborate estimates on solutions of degenerate BSPDEs. 

The study of linear BSPDEs  can date back to  about thirty years ago (see \cite{Bensousan_83}). They arise in many applications of probability theory and stochastic processes, for instance in the nonlinear filtering and stochastic control theory for processes with incomplete information, as an adjoint equation of the Duncan-Mortensen-Zakai filtration equation (for instance, see \cite{Bensousan_83,Hu_Ma_Yong02}). Naturally in the dynamic programming theory, a class of nonlinear BSPDEs as the so-called stochastic Hamilton-Jacobi-Bellman equations, are introduced in the study of non-Markovian control problems (see \cite{Peng_92}). In addition, the representation relationship between forward-backward stochastic differential equations and BSPDEs yields the stochastic Feynman-Kac formula (see \cite{Hu_Ma_Yong02}). The BSPDEs have already received extensive attentions and see \cite{BenderDokuchev-2014,GraeweHorstQui13,Hu_Peng_91,QiuTangMPBSPDE11,QiuWei-RBSPDE-2013,Tang-Wei-2013} and references therein for the recent developments.

The rest of this paper is organized as follows. In section 2, we introduce some notations and show the main result (Theorem \ref{thm-main}). Section 3 is devoted to an $L^2$-theory for linear degenerate BSPDEs. In section 4, we prove the H\"ormander-type theorem.

\section{Preliminaries and main results}

For each $l\in \mathbb{N}^+$ and domain $\Pi\subset \bR^l$, denote by $C_c^{\infty}(\Pi)$ the space of infinitely differentiable functions with compact supports in $\Pi$. $L^2(\bR^d)$ ($L^2$ for short) is the usual Lebesgue integrable space with usual scalar product $\langle\cdot,\,\cdot\rangle$ and norm $\|\cdot\|$.
For  $n\in
(-\infty,\infty)$, we denote by $H ^{n}$ the space of Bessel
potentials, that is
$H ^{n}:=(1-\Delta)^{-\frac{n}{2}} L^{2}$
with the Sobolev norm
$$\|\phi\|_{n}:=\|(1-\Delta)^{\frac{n}{2}}\phi\|_{L^2}, ~~ \phi\in
H^{n}.$$  

For the sake of convenience, we shall also use $\langle \cdot,\,\cdot\rangle$ to denote the duality between $(H^n)^k$ and $(H^{-n})^k$ ($k\in\bN^+,\,n\in\bR$) as well as that between the Schwartz function space $\mathscr{D}$ and  $C_c^{\infty}(\bR^d)$. Moreover, We always omit the index associated to the dimension when there is no confusion.

Given  Banach space ($\mathbb{B}$, $\|\cdot\|_{\mathbb{B}}$), $\cS ^2 (\mathbb{B})$ is the set of all the $\mathbb{B}$-valued,
 $(\sF_t)$-adapted and continuous processes $(X_{t})_{t\in [0,T]}$ such
 that
 $$\|X\|_{\cS ^2(\mathbb{B})}:= \left\| \sup_{t\in [0,T]} \|X_t\|_{\mathbb{B}} \right\|_{L^2(\Omega)}< \infty.$$
  For $p\in[1,\infty]$, denote by $\mathcal{L}^p(\mathbb{B})$ the totality of all the $\mathbb B$-valued,
 $(\sF_t)$-adapted processes $(X_{t})_{t\in [0,T]}$ such
 that
 $$
 \|X\|_{\mathcal{L}^p(V)}:= \left\| \|X_t\|_{\mathbb{B}} \right\|_{L^p(\Omega\times[0,T])}< \infty.
 $$
Obviously, both $(\cS^2(\mathbb{B}),\,\|\cdot\|_{\cS^2(\mathbb{B})})$ and  $(\mathcal{L}^p(\mathbb{B}),\|\cdot\|_{\mathcal{L}^p(\mathbb{B})})$ 
 are Banach spaces.
 
Denote by $C_b$ the space of bounded continuous functions on $\bR^d$ equipped with the usual uniform norm $\|\cdot\|_{\infty}$. Let $C_b^{\infty}$ be the set of infinitely differentiable functions with bounded derivatives of any order.  Denote by $\cL^{\infty}(C_b^{\infty})$ the set of functions $h$ on $\Omega\times [0,T] \times \bR^d$ such that $h(t,x)$ is infinitely differentiable with respect to $x$ and all the derivatives of any order belong to $ \cL^{\infty}(C_b)$.

Throughout this work, we denote $I^n=(1-\Delta)^{\frac{n}{2}}$ for $n\in\bR$. Then $I^n$ belongs to $\Psi_n$ that is the class of pseudo-differential operators of order $n$. By the pseudo-differential operator theory (see \cite{Hormander1983analysis} for instance), the $m$-th order differential operator belongs to $\Psi_m$ for $m\in\bN^+$, the multiplication by elements of $C_b^{\infty}$ lies in $\Psi_0$, and for the reader's convenience, some useful basic results are collected below.
\begin{lem}\label{lem-pdo}
(i). If $J_1\in\Psi_{n_1}$ and $J_2\in\Psi_{n_2}$ with $n_1,n_2\in\bR$, then $J_1J_2\in\Psi_{n_1+n_2}$ and the Lie bracket $[J_1,J_2]:=J_1J_2-J_2J_1\in\Psi_{n_1+n_2-1}$.

(ii). For $m\in (0,\infty)$, let $\zeta$ belong to the continuous function space $C_b^{m}$ which is defined as usual. Then for any $n\in(-m,m)$ there exists constant $C$ such that 
$$
\|\zeta\phi\|_n\leq C \|\zeta\|_{C^{m}}\|\phi\|_n,\quad \forall\,\phi\in H^n.
$$ 
\end{lem}

Set 
$$\bV_0=\{L_1,\dots,L_{d_1}\} \quad \text{and} \quad \bV_{n+1}=\bV_n\cup\{[L_k,V]:\,V\in\bV_n,\,k=1,\dots,d_1\}.$$
Denote by $\bL_n$ the set of linear combinations of elements of $\bV_n$ with coefficients of $\cL^{\infty}(C^{\infty}_b)$. We then introduce the following H\"ormander-type condition.

\medskip
$({\mathcal H} )$ \it 
   There exists $n_0\in\bN_0$ such that $\{D_i:i=1,\dots,d\}\subset \bL_{n_0}$.
   \rm \medskip
   
Throughout this paper, we denote $\eta=2^{-n_0}$.

Instead of BSPDE \eqref{BSPDE}, we consider the following one of the general form
\begin{equation}\label{BSPDE-D-M}
  \left\{\begin{array}{l}
  \begin{split}
  -du(t,x)=\,&\displaystyle \bigg[ \left( \frac{1}{2}L_k^2+\frac{1}{2}M_k^2\right)u+M_kv^k
         +{b}^jD_ju
         +cu+\gamma^lv^l+f+L_kg^k
                \bigg](t,x)\, dt
         \\ &\displaystyle
           -v^{r}(t,x)\, dW_{t}^{r},\quad (t,x)\in Q;\\
    u(T,x)=\, &G(x),\quad x\in\bR^d.
    \end{split}
  \end{array}\right.
\end{equation}
We define the following assumption.

\medskip
   $({\mathcal A} 1)$ \it 
   For $i=1,\dots,d$, $k=1,\dots,d_1$, $\sigma^{ik},\theta^{ik},b^i,\gamma^k, c\in\cL^{\infty}(C_b^{\infty})$.
   \medskip

\begin{defn}\label{defn-solution}
  A pair of processes $(u,v)$
  is called a solution to BSPDE \eqref{BSPDE-D-M}
  if $(u,v)\in \cS^2(H^m)\times\cL^2(H^{m-1})$ for some $m\in\bR$ and BSPDE   \eqref{BSPDE} holds in the distributional sense, i.e.,
   for any $\zeta\in C^{\infty}_c(\bR)\otimes C_c^{\infty}(\bR^d)$ there holds almost surely
  \begin{equation*}
    \begin{split}
      &\langle \zeta(t),\,u(t)\rangle-\langle \zeta(T),\, G\rangle
      +\!\int_t^T\!\!\langle \partial_s \zeta(s),\, u(s) \rangle \, ds
      +\!\int_t^T \!\!\langle \zeta(s),\, v^r(s)\rangle  \,dW_s^r
      \\
      &=
      \int_t^T\!\! \bigg\langle  \zeta,\,
      \frac{1}{2}(L_k^2+M_k^2)u+M_kv^k
         +{b}^jD_ju+cu+\gamma^lv^l+f+L_kg^k
          \bigg\rangle(s)\, ds , \quad \forall \ t\in[0,T].
    \end{split}
  \end{equation*}
\end{defn}

We now state our main result. The following theorem is a summary of Theorem \ref{thm-BSPDE}, Corollary \ref{cor-regularity} and Theorem \ref{thm-hormand}.

\begin{thm}\label{thm-main}
Let assumption $(\cA 1)$ hold. Assume $(f,g,G)\in \cL^2(H^m)\times \cL^2((H^m)^{d_1}) \times L^2(\Omega,\sF_T;H^{m})$, for some $m\in\bR$. There hold the following three assertions:
\medskip \\
(i) 
BSPDE  \eqref{BSPDE-D-M} admits a unique solution $(u,v)\in \cS^2(H^{m})\times\cL^2(H^{m-1})$ with $(L_ku,v^k+M_ku)\in\cL^2(H^m)\times\cL^2(H^m)$, $k=1,\dots,d_1$, and
\begin{align}
&E\sup_{t\in[0,T]}\|u(t)\|_m^2+E\int_{0}^T\left(\sum_{k=1}^{d_1}\|L_ku(t)\|_m^2+\|v(t)+Du(t)\theta(t)\|_m^2\right)dt
\nonumber\\
&\leq C\left\{
E\|G\|_m^2 + E\int_{0}^T \left(\|f(s)\|_m^2+\|g(s)\|_m^2\right)\,ds\right\},\label{estim-thm}
\end{align}
with $C$ depending on $T,m$ and quantities related to coefficients $\sigma,\theta,b,c$ and $\gamma$. 
\medskip \\
(ii) If the H\"ormander-type condition $(\cH)$ holds,   for the above solution $(u,v)$, 
we have further 
\begin{align}
E\int_{0}^T\|u(t)\|_{m+\eta}^2dt
\leq C\left\{
E\|G\|_m^2 + E\int_{0}^T \left(\|f(s)\|_m^2+\|g(s)\|_m^2\right)\,ds\right\},\label{estim-eta-thm}
\end{align}
with $C$ depending on $T,m,n_0,\sigma,\theta,b,c$ and $\gamma$.
\medskip \\
(iii) If both $(f,g)\in \cap_{n\in\bR}\left(\cL^2(H^n)\times\cL^2((H^n)^{d_1})\right)$ and assumption $({\mathcal H} )$ hold, we have for each $\eps\in (0,T)$,
$$(u,v)\in\cap_{n\in\bR} L^2(\Omega;C([0,T-\eps];H^n))\times L^2(\Omega;L^2(0,T-\eps;H^{n-1})),$$
and for any $n\in\bR$
 \begin{align}
 &E\sup_{t\in[0,T-\eps]}\| u(t)\|_{n}^2
+E\int_{0}^{T-\eps}\left(\| u(t)\|_{n+\eta}^2+
\| v(t)+D u(t)\theta(t)\|_{n}^2\right)dt
\nonumber\\
&\leq C\left\{ E\|G\|_m^2+ E\int_{0}^{T} \left(\|f(s)\|_{n}^2+ \|g(s)\|_{n}^2  \right)\,ds\right\},\nonumber
 \end{align}
 with the constant $C$ depending on $\eps,T,n,m,n_0,\sigma,\theta,\gamma,b$ and $c$. In particular, the random field $u(t,x)$ is  infinitely differentiable with respect to $x$ on $[0,T)\times\bR^d$ and each derivative is a continuous function on $[0,T)\times\bR^d$.	
\end{thm}

\begin{rmk}\label{rmk-thm-BSPDE}
An $L^2$-theory on degenerate BSPDEs was initiated by Zhou \cite{Zhou_92}, and it was developed  recently by \cite{DuTangZhang-2013,Hu_Ma_Yong02,ma1999linear}. Along this line, to get a solution of BSPDE \eqref{BSPDE-D-M} in space $\cS^2(H^m)\times\cL^2(H^{m-1})$ requires that $L_kg^k$ lies in $ \cL^2(H^m)$ for some $m\in \bN^+$, but in (i) of Theorem \ref{thm-main}, $g^k$ is allowed to be in $ \cL^2(H^m)$ and thus $L_kg^k\in\cL^2(H^{m-1})$, and there holds the additional gradient estimate $L_ku\in\cL^2(H^m)$, for $k=1,\dots,d_1$.  Moreover, compared with the existing $L^2$-theory on degenerate BSPDEs, $m$ herein can be any real number instead of being restricted to positive integers, and under the H\"ormander-type condition $(\mathcal{H})$, one further has $u\in\cL^2(H^{m+\eta})$ in (ii). Hence, the $L^2$-theory presented in (i) and (ii) of Theorem \ref{thm-main} seems to be of independent interest.

Starting from the estimate of $L_ku$,  we prove the H\"ormander-type theorem ((iii) of Theorem \ref{thm-main}) by increasing the regularity of $u$ step by step. In this paper, it is indeed necessary to allow $m$ to be real number in the $L^2$-theory, as for each step the regularity is increased from $m$ to $m+\eps$ for a possibly  \textit{real}  number $\eps\in (0,1]$ (see Section \ref{sec:proof-main-thm} below for more details). 
\end{rmk}


\section{An $L^2$ theory of linear BSPDEs}

We consider the following BSPDE
\begin{equation}\label{BSPDE-D}
  \left\{\begin{array}{l}
  \begin{split}
  -du(t,x)&= \Big[  \frac{1}{2}\Big(L_k^2+M_k^2\Big) u+M_kv^k
         +{b}^jD_ju
         +cu+\gamma^lv^l+f+L_kg^k
                \Big](t,x)\, dt
         \\ &
           \,\, + \delta \Delta u(t,x)\,dt-v^{r}(t,x)\, dW_{t}^{r},\quad (t,x)\in Q;\\
    u(T,x)&=G(x),\quad x\in \bR^d,
    \end{split}
  \end{array}\right.
\end{equation}
with $\delta \geq 0$.

Note that we do not need the H\"ormander-type condition $(\cH)$ in this section. We would first give an a priori estimate on the solution for BSPDE  \eqref{BSPDE-D}.
\begin{prop}\label{prop-apriori-estm}
Let assumption $(\cA 1)$ hold. For $(f,g,G)\in \cL^2(H^m)\times \cL^2((H^m)^{d_1})\times L^2(\Omega,\sF_T;H^{m})$ with $m\in \bR$, if $(u,v)\in \left(\cS^2(H^{m+1})\cap\cL^2(H^{m+2})\right)\times\cL^2((H^{m+1})^{d_1})$ is a solution of BSPDE \eqref{BSPDE-D}, then one has
\begin{align}
&E\sup_{t\in[0,T]}\|u(t)\|_m^2+E\int_{0}^T\left(\delta\|Du(t)\|_m^2+\sum_{k=1}^{d_1}\|L_ku(t)\|_m^2+\|v(t)+Du(t)\theta(t)\|_m^2\right)dt
\nonumber\\
&\leq C\left\{
E\|G\|_m^2 + E\int_{0}^T \left(\|f(s)\|_m^2+\|g(s)\|_m^2\right)\,ds\right\},\label{estim-aprioi-prop}
\end{align}
with $C$ depends only on $T,m,\sigma,\theta,\gamma,b$ and $c$.
\end{prop}

\begin{proof}
Set $\xi=v+Du\theta$. Putting $L_k':=D_i(\sigma^{ik}\cdot)$ and $M_k'=D_i(\theta^{ik}\cdot)$, we have $L_k=L_k'+c_k$ and $M_k=M_k'+\alpha_k$ with $c_k=-(D_i\sigma^{ik})\cdot$ and $\alpha_k=-(D_i\theta^{ik})\cdot$, for $k=1,\dots,d_1$.
 Applying It\^o formula for the square norm (see e.g. \cite[Theorem 3.1]{Krylov_Rozovskii81}), one has almost surely for $t\in[0,T]$,
\begin{align}
&\|I^mu(t)\|^2+\int_t^T\left(2\delta\|I^mDu(s)\|^2+\|I^m(\xi-Du\theta)(s)\|^2\right)ds
\rangle\nonumber
\\
&=\|I^mG\|^2+
\int_t^T\left\langle
I^mu(s),\,
I^m\left((L_k^2+M_k^2)u+2M_k(\xi^k-D_iu\theta^{ik})
         \right)(s)\right\rangle\,ds
         \nonumber\\
         &\quad
         +\int_t^T2\left\langle I^mu(s),\,I^m\left({b}^jD_ju+\gamma^l(\xi^l-D_iu\theta^{il})+cu+f+L_kg\right)(s)\right\rangle\,ds\nonumber\\
         &-\int_t^T \!\!\!2\langle I^m u(s),\,I^m(\xi-Du\theta)(s)\,dW_s. \label{eq-prop-ito}
\end{align}

First, basic calculations yield
\begin{align}
&\langle I^mu,\,I^m(L_k^2u)\rangle\nonumber\\
&=\langle I^mu,\,I^m(L'_k+c_k)L_ku\rangle\nonumber\\
&=-\langle L_kI^mu,\, I^mL_ku\rangle +\langle I^mu,\,[I^m,L_k']L_ku+ I^mc_kL_ku\rangle\nonumber\\
&=-\| I^mL_ku\|^2 +\langle [I^m,L_k]u,\,I^mL_ku\rangle+\langle I^mu,\,[I^m,L_k']L_ku+ I^mc_kL_ku\rangle\nonumber\\
&\leq-(1-\eps)\|I^mL_ku\|^2+C_{\eps}\|I^mu\|^2,\quad\eps\in(0,1),\label{eq-rela-u}\\
\text{ }\nonumber\\
&\langle I^m u,\,I^m(\gamma^l\xi^l+cu+f+L_kg^k)\rangle\nonumber\\
&=
\langle I^m u,\,I^m(\gamma^l\xi^l+cu+f)\rangle
+\langle I^m u,\,(L_kI^m+[I^m,L_k])g^k\rangle\nonumber\\
&=
 \langle I^m u,\,I^m(\gamma^l\xi^l+cu+f)\rangle
-\langle I^m L_ku+[L_k,I^m]u,\,I^mg^k\rangle
+\langle I^m u,\,(\alpha_kI^m+[I^m,L_k])g^k\rangle\nonumber\\
&\leq
\eps \left(\|I^mL_ku\|^2 +\|I^m\xi^l\|^2 \right)+C_{\eps}\left(\|I^mu\|^2+ \|I^mf\|^2+\|I^mg^k\|^2   \right),\quad \eps\in(0,1),
                                           \label{eq-rela-fgc}
\end{align}
and
\begin{align}
&\langle
I^mu,\,
I^m(M_k^2u+2M_k(\xi^k-D_iu\theta^{ik})
        )\rangle\nonumber\\
&= \langle
I^mu,\,
I^m(-M_k^2u+2M_k\xi^k
        )\rangle\nonumber\\  
 &= \langle
		I^mu,\,-M_k'I^mM_ku+[M_k',I^m]M_ku\rangle+  \langle I^mu,\,I^m(2M_k'\xi^k+2\alpha_k\xi^k-\alpha_kM_ku)
        \rangle\nonumber\\
 &= \langle
		M_kI^mu,\,I^mM_ku\rangle +
		\langle I^mu,\,[M_k',I^m]M_ku\rangle-2\langle M_kI^mu,\,I^m\xi^k\rangle+2\langle I^mu,\,[I^m,M_k']\xi^k\rangle\nonumber\\
	&\quad
		+  
		\langle I^mu,\,I^m(2\alpha_k\xi^k-\alpha_kM_ku)
        \rangle\nonumber\\
 &=\|I^mM_ku\|^2-2\langle I^mM_ku,\,I^m\xi^k\rangle
 +\langle [M_k,I^m]u,\,I^mM_ku\rangle
 +\langle I^mu,\,[M_k',I^m]M_ku\rangle\nonumber\\
 &\quad-2\langle [M_k,I^m]u,\,I^m\xi^k\rangle 
 +2\langle I^mu,\,[I^m,M_k']\xi^k\rangle
 +  
		\langle I^mu,\,I^m(2\alpha_k\xi^k-\alpha_kM_ku)
        \rangle\nonumber\\      
&=\|I^mM_ku\|^2-2\langle I^mM_ku,\,I^m\xi^k\rangle
 -\|[M_k,I^m]u\|^2\nonumber\\
 &\quad
 +\langle [M_k,I^m]u,\,M_kI^mu\rangle
 +\langle I^mu,\,[M_k',I^m]M_ku\rangle
 -\langle I^mu,\,\alpha_kM_kI^mu\rangle
 \nonumber\\
 &\quad-2\langle [M_k,I^m]u,\,I^m\xi^k\rangle 
 +2\langle I^mu,\,[I^m,M_k']\xi^k\rangle
 +  
		\langle I^mu,\,I^m(2\alpha_k\xi^k)+[\alpha_kM_k,I^m]u
        \rangle\nonumber\\
&=\|I^mM_ku\|^2-2\langle I^mM_ku,\,I^m\xi^k\rangle
 -\|[M_k,I^m]u\|^2\nonumber\\
 &\quad
 +\langle I^mu,\,\alpha_k[M_k,I^m]u-  [\alpha_k,I^m]M_ku\rangle
 -\langle I^mu,\,\alpha_kM_kI^mu\rangle
 +\langle I^m u,\,[[M_k,I^m],M_k]u\rangle
 \nonumber\\
 &\quad-2\langle [M_k,I^m]u,\,I^m\xi^k\rangle 
 +2\langle I^mu,\,[I^m,M_k]\xi^k\rangle
 +  
		\langle I^mu,\,2\alpha_kI^m\xi^k+[\alpha_kM_k,I^m]u
        \rangle\nonumber\\  
&\leq \|I^mM_ku\|^2-2\langle I^mM_ku,\,I^m\xi^k\rangle
 -\|[M_k,I^m]u\|^2\nonumber\\
 &\quad+\eps\|I^m\xi\|^2+C_{\eps}\|I^mu\|^2,\quad\eps\in(0,1),      \label{eq-relat-uxi}
\end{align}
where we have used the relation 
\begin{align}
\langle I^mu,\,\alpha_kM_kI^mu\rangle=-\frac{1}{2}\langle I^mu,\,D_i(\alpha_k\theta^{ik})I^mu\rangle.\label{eq-relat-adu}
\end{align}

Noticing relations like \eqref{eq-relat-adu} and  that for $i=1,\dots,d$, $k=1,\dots,d_1$,
\begin{align*}
&\|I^m(\xi^k-M_ku)\|^2=\|I^m\xi^k\|^2-2\langle I^m\xi^k,\,I^mM_ku\rangle+\|I^mM_ku\|^2\\
&\langle I^mu,\,I^m(\gamma^kD_iu\theta^{ik})\rangle 	=\frac{1}{2}
\langle I^mu,\,D_i(\gamma^k\theta^{ik})I^mu
+2[\gamma^k\theta^{ik}D_i,\,I^m]u\rangle,\\
&\langle I^mu,\,I^m(b^iD_iu)\rangle
=-\frac{1}{2}\langle I^mu,\,D_i(b^i)I^mu+2[b^iD_i,I^m]u\rangle,
\end{align*} 
 putting \eqref{eq-prop-ito}, \eqref{eq-rela-u}, \eqref{eq-rela-fgc} and \eqref{eq-relat-uxi} together, and taking expectations on both sides of \eqref{eq-prop-ito}, one gets by Gronwall inequality
\begin{align}
&\sup_{t\in[0,T]}E\|u(t)\|_m^2+E\int_{0}^T\left(\delta\|Du(t)\|_m^2+\sum_{k=1}^{d_1}\|L_ku(t)\|_m^2+\|v(t)+Du(t)\theta(t)\|_m^2\right)dt
\nonumber\\
&\leq C\left\{
E\|G\|_m^2 + E\int_{0}^T \left(\|f(s)\|_m^2+\|g(s)\|_m^2\right)\,ds\right\}. \label{eq-prf-0}
\end{align}
On the other hand, one has for each $t\in[0,T)$
\begin{align*}
&E\sup_{\tau\in[t,T]}\bigg|\int_{\tau}^T \!\!\!2\left\langle I^m u(s),\,I^m(\xi-Du\theta)(s)\,dW_s\right\rangle\bigg|\\
&\leq 2 E\sup_{\tau\in[t,T]}\bigg|\int_t^{\tau} \!\!\!2\left\langle I^m u(s),\,I^m(\xi-Du\theta)(s)\,dW_s\right\rangle\bigg|\\
&\leq C \bigg(E\sum_{k=1}^{d_1}\int_t^T\left(|\langle I^mu(s),\,I^m\xi^k(s)\rangle|^2 +|\langle I^mu(s),\,(M_kI^m+[I^m,M_k])u(s)\rangle|^2\right)ds\bigg)^{1/2}\\
&\leq C \bigg(E\int_t^T
\left(\|I^mu(s)\|^2\|I^m\xi(s)\|^2+\|I^mu(s)\|^4\right)ds
\bigg)^{1/2	}
\\
&\leq \eps E\sup_{s\in[t,T]}\|I^mu(s)\|^2 +C_{\eps} E\int_t^T
\left(\|I^m\xi(s)\|^2+\|I^mu(s)\|^2\right)ds,\quad \forall\,\eps\in(0,1),
\end{align*}
which together with \eqref{eq-prop-ito}, \eqref{eq-rela-u}, \eqref{eq-relat-uxi} and \eqref{eq-prf-0} implies \eqref{estim-aprioi-prop}.

\end{proof}

An immediate consequence of Proposition \ref{prop-apriori-estm} is 
\begin{cor}\label{cor-uniqn}
Let assumption $(\cA 1)$ hold. For $(f,g,G)\in \cL^2(H^m)\times \cL^2((H^m)^{d_1}) \times L^2(\Omega,\sF_T;H^{m})$ with $m\in \bR$, the solution of BSPDE \eqref{BSPDE-D} is unique.
\end{cor}

\begin{thm}\label{thm-BSPDE}
Let assumption $(\cA 1)$ hold. Given $(f,g,G)\in \cL^2(H^m)\times \cL^2((H^m)^{d_1}) \times L^2(\Omega,\sF_T;H^{m})$ with $m\in \bR$, BSPDE \eqref{BSPDE-D-M} (equivalently, BSPDE  \eqref{BSPDE-D} with $\delta=0$) admits a unique solution $(u,v)\in \cS^2(H^{m})\times\cL^2(H^{m-1})$ with $(L_ku,v^k+M_ku)\in\cL^2(H^m)\times\cL^2(H^m)$, $k=1,\dots,d_1$, and
\begin{align}
&E\sup_{t\in[0,T]}\|u(t)\|_m^2+E\int_{0}^T\left(\sum_{k=1}^{d_1}\|L_ku(t)\|_m^2+\|v(t)+Du(t)\theta(t)\|_m^2\right)dt
\nonumber\\
&\leq C\left\{
E\|G\|_m^2 + E\int_{0}^T \left(\|f(s)\|_m^2+\|g(s)\|_m^2\right)\,ds\right\},\label{estim-thm}
\end{align}
with $C$ depending on $T,m,\sigma,\theta,\gamma,b$ and $c$.
\end{thm}

\begin{proof}
We use the method of approximation. Choose $\{\delta_l\}_{l\in\bN^+}\subset(0,1)$ and 
$$
\{(f_n,g_n,G_n)\}_{n\in\bN^+}\subset \cL^2(H^{m+5})\times \cL^2((H^{m+5})^{d_1})\times L^2(\Omega,\sF_T;H^{m+5})
$$
 such that $\delta_l$ converges down to $0$ and $(f_n,g_n,G_n)$ converges to $(f,g,G)$ in $\cL^2(H^m)\times\cL^2((H^{m})^{d_1})\times L^2(\Omega,\sF_T;H^m)$. By the $L^p$-theory of BSPDEs (see \cite{DuQiuTang10} for instance), BSPDE \eqref{BSPDE-D} admits a unique solution $(u_{l,n},v_{l,n})\in \left(\cS^2(H^{m+5})\cap \cL^2(H^{m+6})\right)\times \cL^2(H^{m+5}) $ associated with $(\delta_l,f_n,g_n,G_n)$. 
 
For each $n$, it follows from Proposition \ref{prop-apriori-estm} that  $\{(u_{l,n},L_ku_{l,n},v_{l,n}+Du_{l,n}\theta)\}_{l\in\bN^+}$ is bounded in $\cS^2(H^{m+4})\times \cL^2(H^{m+4})\times\cL^2((H^{m+4})^{d_1})$, $k=1,\dots,d_1$. 
 Notice that $\delta_l\Delta u_{l,n}$ tends to zero in $\cL^2(H^{m+2})$ as $l$ goes to infinity. Therefore, letting $l$ tend to infinity, from Proposition \ref{prop-apriori-estm} and Corollary \ref{cor-uniqn} we derive the unique solution $(u_n,v_n)$ for BSPDE \eqref{BSPDE-D} associated with $(f_n,g_n,G_n)$ and $\delta=0$ such that $(u_n,L_ku_n,v_n+Du_n\theta)\in\cS^{2}(H^{m+2})\times \cL^2(H^{m+2})\times\cL^2((H^{m+2})^{d_1})$ for $k=1,\dots,d_1$. 
 
 Furthermore, letting $n$ go to infinity, again by Proposition \ref{prop-apriori-estm} and Corollary \ref{cor-uniqn}, one obtains the unique solution $(u,v)$ and associated estimates. This completes the proof.
\end{proof}

\begin{rmk}
 Like in \cite{DuTangZhang-2013,Hu_Ma_Yong02}, the random field $v+Du\theta\in \cL^2((H^m)^{d_1})$ is estimated as a unity which appears in the corresponding BSDE (see \eqref{BSDE-relatn} below for instance), and thus we only have $v\in \cL^2(H^{m-1})$ (see Example \ref{exam}). In fact, if we further have $\sigma\sigma^{\mathcal{T}}\geq \theta\theta^{\mathcal{T}}$, then $Du\theta\in\cL^2(H^{m})$ and thus $v\in\cL^2(H^{m})$, as $u,L_ku\in\cL^2(H^m)$, $k=1,\dots,d_1$. In addition, in view of (ii) of Lemma \ref{lem-pdo} and the proofs involved in this section, the required  regularity for the coefficients $b$, $c$, $\sigma$, $\theta$ and $\gamma$ can be relaxed like in  \cite{DuTangZhang-2013,Hu_Ma_Yong02}, but we would not seek such a generality in the present paper.   
\end{rmk}
\section{H\"ormander-type theorem}
\label{sec:proof-main-thm}
Recall that $\eta=2^{-n_0}$. Basing on the $L^2$-theory of SPDEs presented in the preceding section, we derive the following H\"ormander-type theorem.
\begin{thm}\label{thm-hormand}
Let assumptions $(\cA 1)$ and $(\cH)$ hold. Suppose that 
$$(f,g)\in \cap_{n\in\bR}\left(\cL^2(H^n)\times\cL^2((H^n)^{d_1})\right)\quad \text{and}\quad G\in L^2(\Omega;H^m) \quad\text{for some }m\in\bR.$$ For the unique solution $(u,v)$ of BSPDE \eqref{BSPDE-D-M} in Theorem \ref{thm-BSPDE},  we have for any $\eps\in (0,T)$
$$(u,v)\in\cap_{n\in\bR} L^2(\Omega;C([0,T-\eps];H^n))\times L^2(\Omega;L^2(0,T-\eps;H^{n-1})),$$
and for any $n\in\bR$
 \begin{align}
 &E\sup_{t\in[0,T-\eps]}\| u(t)\|_{n}^2
+E\int_{0}^{T-\eps}\left(\| u(t)\|_{n+\eta}^2+
\| v(t)+D u(t)\theta(t)\|_{n}^2\right)dt
\nonumber\\
&\leq C\left\{ E\|G\|_m^2+ E\int_{0}^{T} \left(\|f(s)\|_{n}^2+ \|g(s)\|_{n}^2  \right)\,ds\right\},\label{est-hmd-thm}
 \end{align}
 with the constant $C$ depending on $\eps,T,n,m,n_0,\sigma,\theta,\gamma,b$ and $c$.
 In particular, the random field $u(t,x)$ is almost surely infinitely differentiable with respect to $x$ on $[0,T)\times\bR^d$ and each derivative is a continuous function on $[0,T)\times\bR^d$.	
\end{thm}
 Because of the appearance of the stochastic integral in BSPDE \eqref{BSPDE-D-M}, we do not investigate the time-differentiability of $u(t,x)$ and the coefficients herein is only required to be measurable with respect to the time variable, while in the classical H\"ormander theorem, the associated coefficients are smooth and the function $u(t,x)$ turns out to be deterministic and smooth with respect to the time variable.

Before the proof of Theorem \ref{thm-hormand}, we first give an estimate on the Lie bracket.
 
\begin{lem}\label{lem-lie-bracket}
For $\{\widetilde{L},L\}\subset\cup_{l\geq 0} \bV_l$, $m\in\bR$ and $\eps\in[0,1]$, there exists a positive constant $C$ such that almost surely for any $\phi\in H^{m}$ with $\widetilde{L}\phi\in H^{m-1+\eps}$ and $L\phi\in H^m$,
\begin{align*}
\|[\widetilde{L},L]\phi\|_{m-1+\frac{\eps}{2}}
\leq C\left(
\|\widetilde{L}\phi\|_{m-1+\eps}+\|L\phi\|_{m}+\|\phi\|_m
\right).
\end{align*}
\end{lem}

\begin{proof}
Assume first $\phi\in H^{m+1}$. Setting $A^n=I^{n-1}[\widetilde{L},L]$, we have $A^n\in\Psi_{n}$ a.s., for each $n\in\bR$.  As the joint operators of $H$ and $L$, $\widetilde{L}^*=-\widetilde{L}+\tilde{c}$ and $L^*=-L+\bar{c}$ with $\tilde{c},\bar{c} \in \cL^{\infty}(C^{\infty}_b)$ respectively. By Lemma \ref{lem-pdo}, one has
\begin{align*}
&\langle \widetilde{L}L\phi,\,I^mA^{m-1+\eps}\phi \rangle\\
&=\langle L\phi,\,(I^m\widetilde{L}^*+[\widetilde{L}^*,I^m])A^{m-1+\eps}\phi\rangle\\
&=\langle I^mL\phi,\,(A^{m-1+\eps}\widetilde{L}^*+[\widetilde{L}^*,A^{m-1+\eps}])\phi  \rangle 
+\langle [I^m,\widetilde{L}]L\phi,\,A^{m-1+\eps}\phi\rangle\\
&\leq C\left( \|L\phi\|_m^2+ \|\widetilde{L}\phi\|^2_{m-1+\eps}+\|\phi\|_m^2   \right)
\end{align*}
and 
\begin{align*}
&\langle \widetilde{L}\phi,\,I^mA^{m-1+\eps}\phi \rangle\\
&=\langle \widetilde{L}\phi,\,(I^{m-1+\eps}L^*+[L^*,I^{m-1+\eps}])A^m\phi\rangle\\
&=\langle I^{m-1+\eps}\widetilde{L}\phi,\,(A^mL^*+[L^*,A^m])\phi\rangle
+\langle I^{m-1+\eps}\widetilde{L}\phi,\, I^{-(m-1+\eps)}[L^*,I^{m-1+\eps}]A^m\phi\rangle\\
&\leq
C\left(\|\widetilde{L}\phi\|_{m-1+\eps}^2+\|L\phi\|_m^2+\|\phi\|_m^2 \right).
\end{align*}
Hence, 
\begin{align*}
\|[\widetilde{L},L]\phi\|_{m-1+\frac{\eps}{2}}
=\langle [\widetilde{L},L]\phi,\,I^mA^{m-1+\eps}\phi \rangle^{\frac{1}{2}}
\leq C\left(\|\widetilde{L}\phi\|_{m-1+\eps}+\|L\phi\|_m+\|\phi\|_m \right).
\end{align*}
Through standard density arguments, one verifies that the above estimate also holds for any $\phi\in H^{m}$ with $\widetilde{L}\phi\in H^{m-1+\eps}$ and $L\phi\in H^m$.  
\end{proof}

Starting from estimate \eqref{estim-thm} of Theorem \ref{thm-BSPDE}, applying Lemma \ref{lem-lie-bracket} iteratively  to elements of  $\bV_0,\dots,\bV_{n_0}$, we have
\begin{cor}\label{cor-regularity}
Assume the same hypothesis as in Theorem \ref{thm-BSPDE}. Let condition $(\cH)$ hold. For the unique solution $(u,v)$ of BSPDE \eqref{BSPDE-D-M}, one has further  $u\in \cL^2(H^{m+\eta})$ with
\begin{align}
E\int_{0}^T\|u(t)\|_{m+\eta}^2dt
\leq C\left\{
E\|G\|_m^2 + E\int_{0}^T \left(\|f(s)\|_m^2+\|g(s)\|_m^2\right)\,ds\right\},\label{estim-eta-thm}
\end{align}
where the constant $C$ depends on $T,m,n_0,\sigma,\theta,b,c$ and $\gamma$.
\end{cor}

We are now ready to present the proof of Theorem \ref{thm-hormand}.
\begin{proof}[Proof of Theorem \ref{thm-hormand}]
By Theorem \ref{thm-BSPDE}, BSPDE \eqref{BSPDE-D-M} admits a unique solution $(u,v)\in\cS^2(H^m)\times \cL^2(H^{m-1})$ and the pair of random fields $(\bar{u},\bar{v})(t,x):=(T-t)(u,v)(t,x)$ turns out to be the unique solution of BSPDE
\begin{equation}\label{BSPDE-1}
  \left\{\begin{array}{l}
  \begin{split}
  -d\bar{u}(t,x)=\,&\displaystyle \bigg[ \left( \frac{1}{2}L_k^2+\frac{1}{2}M_k^2\right)\bar{u}+M_k\bar{v}^k
         +{b}^jD_j\bar{u} + c\bar u+\gamma^l\bar v^l+(T-t)(f+L_kg^k)
         \\ &\displaystyle
          +u 
                \bigg](t,x)\, dt-\bar{v}^{r}(t,x)\, dW_{t}^{r};\\
    \bar{u}(T,x)=\, &0,
    \end{split}
  \end{array}\right.
\end{equation}
with 
\begin{align*}
&E\sup_{t\in[0,T]}\|\bar u(t)\|_m^2+E\int_{0}^T\left(\sum_{k=1}^{d_1}\|L_k\bar u(t)\|_m^2+\|\bar v(t)+D\bar u(t)\theta(t)\|_m^2\right)dt
\nonumber\\
&\leq C\left(T^2+1\right)E\int_{0}^T \left(\|f(s)\|_m^2+\|g(s)\|^2_m+\|u(s)\|_m^2\right)\,ds.
\end{align*}
Starting from the above estimate and applying Lemma \ref{lem-lie-bracket} iteratively  to elements of  $\bV_0,\dots,\bV_{n_0}$, we have 
\begin{align*}
\int_t^T\|D\bar{u}\|^2_{m-1+\eta}ds\leq 
C\left(T^2+1\right)E\int_{0}^T \left(\|f(s)\|_m^2+\|g(s)\|^2_m+\|{u}(s)\|_m^2\right)\,ds.
\end{align*}
 Fix any $\eps\in(0,T\wedge 1)$ and define $\eps_l=\sum_{i=1}^l\frac{\eps}{2^i}$. 
By interpolation and Theorem \ref{thm-BSPDE}, one gets 
\begin{align}
&E\sup_{t\in[0,T-\eps_1]}\| u(t)\|_m^2
+E\int_{0}^{T-\eps_1}\left(\| u(t)\|_{m+\eta}^2+
\| v(t)+D u(t)\theta(t)\|_m^2\right)dt
\nonumber\\
&\leq \frac{C2(T^2+1)}{\eps}E\int_{0}^{T} \left(\|f(s)\|_m^2+\|g(s)\|_m^2+\|u(s)\|_m^2\right)\,ds.
\end{align}

Noticing that $(f,g)\in \cap_{n\in\bR}\left(\cL^2(H^n)\times\cL^2((H^n)^{d_1})\right)$, by iteration we obtain for any $j\in\bN^+$,
\begin{align}
&E\sup_{t\in[0,T-\eps_j]}\| u(t)\|_{m+(j-1)\eta}^2
+E\int_{0}^{T-\eps_j}\left(\| u(t)\|_{m+j\eta}^2+
\| v(t)+D u(t)\theta(t)\|_{m+(j-1)\eta}^2\right)dt
\nonumber\\
&\leq \frac{C2^j(T^{2}+1)}{\eps}E\int_{0}^{T-\eps_{j-1}} \left(\|f(s)\|_{m+(j-1)\eta}^2+ \|g(s)\|_{m+(j-1)\eta}^2+\|u(s)\|_{m+(j-1)\eta}^2  \right)\,ds,
\end{align}
which, together with estimate \eqref{estim-thm}, implies by iteration that 
\begin{align}
&E\sup_{t\in[0,T-\eps_j]}\| u(t)\|_{m+(j-1)\eta}^2
+E\int_{0}^{T-\eps_j}\left(\| u(t)\|_{m+j\eta}^2+
\| v(t)+D u(t)\theta(t)\|_{m+(j-1)\eta}^2\right)dt
\nonumber\\
&\leq C_j\left\{ E\|G\|_m^2+ E\int_{0}^{T} \left(\|f(s)\|_{m+(j-1)\eta}^2+ \|g(s)\|_{m+(j-1)\eta}^2  \right)\,ds\right\}.\nonumber
\end{align}
Hence, we have
$$(u,v)\in\cap_{n\in\bR} L^2(\Omega;C([0,T-\eps];H^n))\times L^2(\Omega;L^2(0,T-\eps;H^{n-1})),\quad \forall\, \eps\in (0,T),$$
and there holds estimate \eqref{est-hmd-thm}. In particular, by Sobolev embedding theorem, the random field $u(t,x)$ is almost surely infinitely differentiable with respect to $x$ and each derivative is a continuous function on $[0,T)\times\bR^d$.

\end{proof}

At the end of this section, we would show the connection between the conditional expectation \eqref{eq-conditn-exp} and the solution of BSPDE \eqref{BSPDE}.

\begin{prop}\label{prop-repsn}
For the coefficients $G,f,\sigma,\theta,b$, we assume the same hypothesis of Theorem \ref{thm-hormand}. Suppose further that $G\in L^p(\Omega;C_b)$.  Let $(u,v)\in\cS^2(H^m)\times \cL^2(H^{m-1})$ be the solution of BSPDE \eqref{BSPDE}. Then, we have for all $x\in\bR^d$,
\begin{equation}\label{eq-prop-repsn}
u(t,X_t^{s,x})
=E_{\bar{\sF}_t}\left[ G(X_T^{s,x})
+\int_t^T f(r,X_r^{s,x})\,dr\right]\quad \text{a.s.},\,\, \text{for all }0\leq s\leq t\leq T.
\end{equation}
\end{prop}
\begin{proof}
In view of the continuity of $X_r^{s,x}$ with respect to $(s,x,r)$, we first check that all the terms involved in relation \eqref{eq-prop-repsn} make sense in view of the H\"ormander-type theorem \ref{thm-hormand}. 
Let $\rho\in C_c^{\infty}(\bR^d)$ be a nonnegative function with the support in the unit ball centered at the origin such that $\int_{\bR^d} \rho(y)\,dy=1$.  Define the convolution:
$$G_N(x)=\int_{\bR^d}G(x-y)\rho\left(Ny\right)N^d\,dy,\quad \text{for }N\in\bN^+.$$
 In view of the smooth approximation of identity, we have $G_N\in \cap_{n\in\bR} L^2(\Omega;H^n)$ for each $N\in\bN^+$, and $G_N$ converges to $G_N$ in both spaces $L^{2}(\Omega;H^m)$ and $L^2(\Omega;C_b)$. Obviously, it holds that $\lim_{N\rightarrow \infty}E\|G_N(X_T^{s,\cdot})-G(X_T^{s,\cdot})\|_{C_b}^2=0$. For each $N$, let $(u_N,v_N)$ be the unique solution of BSPDE \eqref{BSPDE} with $G$ replaced by $G_N$.
For each $t\in[0,T)$, by Theorem \ref{thm-hormand}, we have for any $\eps\in [0,T-t)$,
$$
(u_N,v_N), (u,v)\in \cap_{n\in\bR} L^2(\Omega;C([0,t+\eps];H^n))\times L^2(\Omega;L^2(0,t+\eps;H^{n-1})),$$
and
$$
\|u_N-u\|^2_{L^2(\Omega;C([0,t+\eps];H^n))}
\leq 
C(n)E\|G_N-G\|_m^2\rightarrow 0,\quad \text{as }N\rightarrow\infty,\quad \forall\,n\in\bR,
$$
and in particular, since $H^{d+2}$ is embedded into $C_b$, there holds
$$
E\sup_{r\in[s,t+\eps]}\|(u_N-u)(r,X_r^{s,\cdot})\|_{C_b}^2\leq
C  \|u_N-u\|^2_{L^2(\Omega;C([0,t+\eps];H^{d+2}))}\rightarrow 0,\quad \text{as }N\rightarrow\infty.
$$

On the other hand, by the It\^o-Kunita formula we have for each $N$ and $0\leq s\leq t$,
\begin{align}
u_N(t,X_t^{s,x})
&= G_N(X_T^{s,x})
+\int_t^T f(r,X_r^{s,x})\,dr
-\int_t^T (v_N+Du_N\theta)(r,X_r^{s,x})\,dW_r\nonumber\\
&\quad-\int_t^T Du_N\sigma (r,X_r^{s,x})\,dB_r\quad \text{a.s.},\quad\forall\,x\in\bR^d.\label{BSDE-relatn}
\end{align}
Taking conditional expectations on both sides, we get for every $x\in\bR^d$
\begin{equation}\label{eq-prop-repsn-N}
u_N(t,X_t^{s,x})
=E_{\bar{\sF}_t}\left[ G_N(X_T^{s,x})
+\int_t^T f(r,X_r^{s,x})\,dr\right]\quad \text{a.s.},\quad \text{for all }0\leq s\leq t\leq T.
\end{equation}
Letting $N$ goes to infinity, we prove \eqref{eq-prop-repsn}.
\end{proof}
\begin{rmk}\label{rmk-measurablity}
In Proposition \ref{prop-repsn}, we assume $G\in L^2(\Omega;C_b)$ to make sense of the composition $G(X_T^{t,x})$.   We would also remark that by taking $s=t$ in relation \eqref{eq-prop-repsn}, the function $u(t,x)$ defined by \eqref{eq-conditn-exp} is just $\sF_t$-measurable and that the conditional expectation in \eqref{eq-conditn-exp} is equivalent to the one with respect to the sub-filtration $\{\sF_t\}_{t\geq 0}$, i.e.,
\begin{align*}
u(t,x)=E_{{\sF}_t}\left[
\int_t^Tf(r,X_r^{t,x})\,dr
+G(X_T^{t,x})
\right],\quad (t,x)\in[0,T]\times\bR^d. \label{eq-conditn-exp}
\end{align*}

\end{rmk}


\bibliographystyle{siam}
%

\begin{thebibliography}{10}

\bibitem{BenderDokuchev-2014}
{\sc C.~Bender and N.~Dokuchaev}, {\em A first-order {BSPDE} for swing option
  pricing}, Math. Finance,  (2014).
\newblock DOI: 10.1111/mafi.12067.

\bibitem{Bensousan_83}
{\sc A.~Bensoussan}, {\em Maximum principle and dynamic programming approaches
  of the optimal control of partially observed diffusions}, Stoch., 9 (1983),
  pp.~169--222.

\bibitem{CassFriz-2010}
{\sc T.~CASS and P.~FRIZ}, {\em Densities for rough differential equations
  under {H{\"o}rmander's} condition}, Ann. Math., 171 (2010), pp.~2115--2141.

\bibitem{DuQiuTang10}
{\sc K.~Du, J.~Qiu, and S.~Tang}, {\em $\textrm{L}^p$ theory for
  super-parabolic backward stochastic partial differential equations in the
  whole space}, Appl. Math. Optim., 65 (2011), pp.~175--219.

\bibitem{DuTangZhang-2013}
{\sc K.~Du, S.~Tang, and Q.~Zhang}, {\em $\textrm{W}^{m,p}$-solution ($p\ge 2$)
  of linear degenerate backward stochastic partial differential equations in
  the whole space}, J. Differ. Equ., 254 (2013), pp.~2877--2904.

\bibitem{GraeweHorstQui13}
{\sc P.~Graewe, U.~Horst, and J.~Qiu}, {\em A non-markovian liquidation problem
  and backward {SPDEs} with singular terminal conditions}.
\newblock to appear in SIAM J. Control Optim., 2014.

\bibitem{hormander-HypoEllip1967}
{\sc L.~H{\"o}rmander}, {\em Hypoelliptic second order differential equations},
  Acta Math., 119 (1967), pp.~147--171.

\bibitem{Hormander1983analysis}
\leavevmode\vrule height 2pt depth -1.6pt width 23pt, {\em The Analysis of
  Linear Partial Differential Operators {III}}, vol.~257, Springer, 1983.

\bibitem{Hu_Ma_Yong02}
{\sc Y.~Hu, J.~Ma, and J.~Yong}, {\em On semi-linear degenerate backward
  stochastic partial differential equations}, Probab. Theory Relat. Fields, 123
  (2002), pp.~381--411.

\bibitem{Hu_Peng_91}
{\sc Y.~Hu and S.~Peng}, {\em Adapted solution of a backward semilinear
  stochastic evolution equations}, Stoch. Anal. Appl., 9 (1991), pp.~445--459.

\bibitem{krylov2013-Hormder-SPDE}
{\sc N.~V. Krylov}, {\em H\"ormander's theorem for stochastic partial
  differential equations}, arXiv preprint arXiv:1309.5543,  (2013).

\bibitem{Krylov_Rozovskii81}
{\sc N.~V. Krylov and B.~L. Rozovskii}, {\em Stochastic evolution equations},
  J. Sov. Math., 16 (1981), pp.~1233--1277.

\bibitem{ma1999linear}
{\sc J.~Ma and J.~Yong}, {\em On linear, degenerate backward stochastic partial
  differential equations}, Probab. Theory Relat. Fields, 113 (1999),
  pp.~135--170.

\bibitem{Malliavin-1978}
{\sc P.~Malliavin}, {\em Stochastic calculus of variation and hypoelliptic
  operators}, in Proc. Intern. Symp. SDE Kyoto 1976, Kinokuniya, Tokyo and
  Wiley, New York, 1978, pp.~195--263.

\bibitem{Mattingly-Pardoux-2006}
{\sc J.~C. Mattingly and {\'E}.~Pardoux}, {\em Malliavin calculus for the
  stochastic {2D Navier-Stokes} equation}, Commun. Pure Appl. Math., 59 (2006),
  pp.~1742--1790.

\bibitem{Peng_92}
{\sc S.~Peng}, {\em Stochastic {H}amilton-{J}acobi-{B}ellman equations}, {SIAM}
  J. Control Optim., 30 (1992), pp.~284--304.

\bibitem{QiuTangMPBSPDE11}
{\sc J.~Qiu and S.~Tang}, {\em Maximum principles for backward stochastic
  partial differential equations}, J. Funct. Anal., 262 (2012), pp.~2436--2480.

\bibitem{QiuWei-RBSPDE-2013}
{\sc J.~Qiu and W.~Wei}, {\em On the quasi-linear reflected backward stochastic
  partial differential equations}, J. Funct. Anal., 267 (2014), pp.~3598--3656.

\bibitem{Stroock-Varadhan-Multi-dim-diffun79}
{\sc D.~W. Stroock and S.~R.~S. Varadhan}, {\em Multidimensional Diffusion
  Processes}, vol.~233, Springer, New York, 1979.

\bibitem{Tang-Wei-2013}
{\sc S.~Tang and W.~Wei}, {\em On the {Cauchy} problem for backward stochastic
  partial differential equations in {H\"{o}lder} spaces},  (2014).
\newblock to appear in {Ann. Probab.}, arXiv:1304.5687v1 [math.AP].

\bibitem{Zhou_92}
{\sc X.~Zhou}, {\em A duality analysis on stochastic partial differential
  equations}, J. Funct. Anal., 103 (1992), pp.~275--293.

\end{thebibliography}

\end{document}